\documentclass[12pt]{article}

\usepackage[T1]{fontenc}
\usepackage[utf8]{inputenc}
\usepackage{lmodern}

\usepackage{mathtools,amssymb,amsthm}
\usepackage{booktabs}
\usepackage{graphicx}
\usepackage{hyperref}

\numberwithin{equation}{section}

\newtheorem{theorem}{Theorem}
\newtheorem{lemma}{Lemma}
\newtheorem{corollary}{Corollary}
\newtheorem{proposition}{Proposition}
\theoremstyle{remark}
\newtheorem{remark}{Remark}

\newcommand{\Mob}{\mu_{\mathrm M}}

\usepackage{geometry}
\title{Thinned Wallis-type prime products in residue classes modulo $2^m$}

\author{Mike Winkler\\[1mm]
        \small Fakult\"at f\"ur Mathematik, Ruhr-Universit\"at Bochum, Germany\\ \small mike.winkler@ruhr-uni-bochum.de
        }

\date{\small June 20, 2026}

\begin{document}

\maketitle

\begin{abstract}
	For odd primes $p$ we consider the factors
	\[
		A(p)=\frac{p-\chi_4(p)}{p+\chi_4(p)},
		\qquad
		\chi_4(p)=
		\begin{cases}
			 \ \ 1,&p\equiv 1\pmod 4,\\
			    -1,&p\equiv 3\pmod 4,
		\end{cases}
	\]
	and study products of $A(p)$ restricted to unions of residue classes modulo $2^m$. We give a simple criterion for the existence of a finite nonzero limit, prove a logarithmic asymptotic in the general case, express the limiting constant in terms of Mertens-type constants in arithmetic progressions (hence in terms of Dirichlet $L$-values), and give reproducible high-precision computations of the resulting constants.
\end{abstract}

%\vspace{3mm}
%\noindent\textit{2020 Mathematics Subject Classification.} 11N37, 11N13 (primary); 11M20, 11M06 (secondary).
%
%\vspace{2mm}
%\noindent\textit{Key words and phrases.} prime products; Mertens constants; arithmetic progressions; Dirichlet characters; Euler products.

Define $A(p)$ as in the abstract. The full product over odd primes has a finite limit:
\[
\prod_{p\ \mathrm{odd}} A(p)=2.
\]
The classical Wallis product is
\[
\frac{\pi}{2}=\prod_{n=1}^{\infty}\frac{(2n)^2}{(2n-1)(2n+1)}.
\]
Equivalently, after indexing by odd integers $r>1$, it can be written as
\[
\frac{\pi}{2}=\prod_{\substack{r>1\\ r\ \mathrm{odd}}}\frac{r-\chi_4(r)}{r+\chi_4(r)}.
\]
We call the factors in this odd-integer formulation the Wallis factors. Thus, $A(p)$ is obtained by replacing the odd integer $r$ in a Wallis factor by an odd prime $p$.
Thinning by congruence conditions modulo $2^m$ is natural here, since $\chi_4$ is a character modulo $4$ and thus constant on residue classes modulo $2^m$. By \textit{thinning} we simply mean restricting the product to primes in a prescribed set of residue classes. The aim is to understand what happens after restricting $p$ to selected residue classes modulo $q=2^m$: for which selections does the product still converge to a finite nonzero limit?

\section{Notation and Preliminary Facts}

Fix $q=2^m$ with $m\ge 2$. Let $(\mathbb Z/q\mathbb Z)^\times$ be the group of reduced residue classes modulo $q$. Throughout, $p$ denotes an odd prime. For $S\subset(\mathbb Z/q\mathbb Z)^\times$ define
\[
P_S(x)=\prod_{\substack{p\le x\\ p\bmod q\in S}} A(p).
\]
Since $4$ divides $q$, each $a\in(\mathbb Z/q\mathbb Z)^\times$ has a well-defined value $\chi_4(a)\in\{\pm 1\}$. We use Mertens' theorem in arithmetic progressions in the form
\[
\sum_{\substack{p\le x\\ p\equiv a\ (q)}} \frac{1}{p}
=\frac{1}{\varphi(q)}\log\log x + B(q,a) + o(1),
\qquad x\to\infty,
\]
for each $a\in(\mathbb Z/q\mathbb Z)^\times$, with constants $B(q,a)$ depending on $q$ and $a$; see \cite{Williams}.

\section{A Structural Identity}

The following elementary factorization is useful throughout the paper.
\begin{lemma}\label{lem:factor}
For every odd prime $p$ one has
\[
A(p)=\frac{(1-\chi_4(p)/p)^2}{1-1/p^2}.
\]
\end{lemma}

\begin{proof}
We compute
\[
\frac{p-\chi_4(p)}{p+\chi_4(p)}
=\frac{(p-\chi_4(p))^2}{p^2-\chi_4(p)^2}
=\frac{(1-\chi_4(p)/p)^2}{1-1/p^2},
\]
since $\chi_4(p)^2=1$ for odd primes $p$.
\end{proof}

\section{Asymptotic Behavior and the Balance Criterion}

For $S\subset(\mathbb Z/q\mathbb Z)^\times$ set
\[
\mu(S)=\sum_{a\in S}\chi_4(a).
\]
The main asymptotic statement is as follows.
\begin{theorem}\label{thm:asymptotic}
There exists $K(q,S)\in(0,\infty)$ such that
\[
P_S(x)=K(q,S)\,(\log x)^{-2\mu(S)/\varphi(q)}(1+o(1)),
\qquad x\to\infty.
\]
\end{theorem}

\begin{proof}
Write
\[
A(p)=\frac{p-\chi_4(p)}{p+\chi_4(p)}=\frac{1-\chi_4(p)/p}{1+\chi_4(p)/p}.
\]
With $u=\chi_4(p)/p$ we have $|u|\le 1/3$ for $p\ge 3$, hence
\[
\log\frac{1-u}{1+u}=-2u+O(u^3),
\qquad u\to 0,
\]
and therefore
\[
\log A(p)=-\frac{2\chi_4(p)}{p}+O\left(\frac{1}{p^3}\right).
\]
Summing over primes $p\le x$ with $p\bmod q\in S$ yields
\[
\log P_S(x)
= -2\sum_{\substack{p\le x\\ p\bmod q\in S}}\frac{\chi_4(p)}{p}
+O\left(\sum_p\frac{1}{p^3}\right)
= -2\sum_{\substack{p\le x\\ p\bmod q\in S}}\frac{\chi_4(p)}{p}
+O(1),
\]
since $\sum_p p^{-3}$ converges. Decompose the main sum into residue classes:
\[
\sum_{\substack{p\le x\\ p\bmod q\in S}}\frac{\chi_4(p)}{p}
=\sum_{a\in S}\chi_4(a)\sum_{\substack{p\le x\\ p\equiv a\ (q)}}\frac{1}{p}.
\]
By Mertens' theorem in arithmetic progressions,
\[
\sum_{\substack{p\le x\\ p\equiv a\ (q)}}\frac{1}{p}
=\frac{1}{\varphi(q)}\log\log x + B(q,a)+o(1),
\qquad x\to\infty,
\]
for each $a\in(\mathbb Z/q\mathbb Z)^\times$. Hence,
\[
\log P_S(x)= -\frac{2\mu(S)}{\varphi(q)}\log\log x + C(q,S)+o(1),
\]
where $C(q,S)=-2\sum_{a\in S}\chi_4(a)B(q,a)+O(1)$ is finite. Exponentiating gives
\[
P_S(x)=K(q,S)\,(\log x)^{-2\mu(S)/\varphi(q)}(1+o(1)),
\]
with $K(q,S)=e^{C(q,S)}\in(0,\infty)$.
\end{proof}

The preceding theorem immediately gives the convergence criterion.
\begin{corollary}\label{cor:criterion}
The limit $\lim_{x\to\infty}P_S(x)$ exists in $(0,\infty)$ if and only if $\mu(S)=0$.
\end{corollary}

\section{The Limiting Constant}

For each $a\in(\mathbb Z/q\mathbb Z)^\times$ define
\[
C(q,a)=\lim_{x\to\infty}(\log x)^{1/\varphi(q)}
\prod_{\substack{p\le x\\ p\equiv a\ (q)}}\left(1-\frac{1}{p}\right),
\]
and
\[
D(q,a)=\prod_{\substack{p\\ p\equiv a\ (q)}}\left(1-\frac{1}{p^2}\right).
\]
Note that $D(q,a)=\zeta_{q,a}(2)^{-1}$, where $\zeta_{q,a}(s)=\prod_{p\equiv a\ (q)}(1-p^{-s})^{-1}$. This normalization matches the standard Mertens product asymptotic in arithmetic progressions; see \cite{Williams}.
We shall use the following formula of Williams.
\begin{proposition}[Williams {\cite[Theorem 1]{Williams}}]\label{prop:williamsC}
For each $a\in(\mathbb Z/q\mathbb Z)^\times$ one has
\[
C(q,a)=\left(e^{-\gamma}\frac{q}{\varphi(q)}
\prod_{\substack{\chi\bmod q\\ \chi\ne\chi_0}}
\left(\frac{K_{\mathrm W}(1,\chi)}{L(1,\chi)}\right)^{\overline{\chi}(a)}
\right)^{1/\varphi(q)},
\]
where $\chi_0$ denotes the principal character, $\gamma$ is Euler's constant, and $K_{\mathrm W}(1,\chi)$ is the nonzero constant appearing in Williams' theorem.
\end{proposition}

The next result identifies the constant $K(q,S)$ from Theorem~\ref{thm:asymptotic}. The balance condition is not needed for this identification; it is needed only for the disappearance of the logarithmic factor. The analytic framework for the Mertens product constants in arithmetic progressions is developed by Languasco and Zaccagnini \cite{LZ,LZII}. Efficient algorithms for high-precision evaluation of $C(q,a)$ are given in \cite[(5)]{LZII}; for $\zeta_{q,a}(2)$, and hence for $D(q,a)$, see Languasco and Moree \cite[(5.16)]{LM}.
\begin{theorem}\label{thm:constant}
For every subset $S\subset(\mathbb Z/q\mathbb Z)^\times$ one has
\[
P_S(x)=
\left(
\prod_{a\in S}
\left(\frac{C(q,a)^2}{D(q,a)}\right)^{\chi_4(a)}
\right)
(\log x)^{-2\mu(S)/\varphi(q)}(1+o(1)).
\]
In particular,
\[
K(q,S)=
\prod_{a\in S}
\left(\frac{C(q,a)^2}{D(q,a)}\right)^{\chi_4(a)}.
\]
If $\mu(S)=0$, then
\[
\lim_{x\to\infty}P_S(x)
=\prod_{a\in S}\left(\frac{C(q,a)^2}{D(q,a)}\right)^{\chi_4(a)}.
\]
\end{theorem}

\begin{proof}
By Lemma~\ref{lem:factor},
\[
A(p)=\frac{(1-\chi_4(p)/p)^2}{1-1/p^2},
\]
so
\[
P_S(x)=
\Biggl(\prod_{\substack{p\le x\\ p\bmod q\in S}}\Bigl(1-\frac{\chi_4(p)}{p}\Bigr)\Biggr)^2
\cdot
\prod_{\substack{p\le x\\ p\bmod q\in S}}\Bigl(1-\frac{1}{p^2}\Bigr)^{-1}.
\]
Fix a reduced residue class $a$ modulo $q$. Since $4$ divides $q$, the value $\chi_4(p)$ is constant on the progression $p\equiv a\pmod q$, namely $\chi_4(p)=\chi_4(a)$. Moreover,
\[
1-\frac{\chi_4(a)}{p}=
\begin{cases}
1-\dfrac1p, & \chi_4(a)=1,\\[2mm]
1+\dfrac1p=\dfrac{1-1/p^2}{1-1/p}, & \chi_4(a)=-1.
\end{cases}
\]
Consequently, for primes $p\equiv a\pmod q$,
\[
\prod_{\substack{p\le x\\ p\equiv a\ (q)}}\Bigl(1-\frac{\chi_4(a)}{p}\Bigr)
=
\left(\prod_{\substack{p\le x\\ p\equiv a\ (q)}}\Bigl(1-\frac1p\Bigr)\right)^{\chi_4(a)}
\cdot
\left(\prod_{\substack{p\le x\\ p\equiv a\ (q)}}\Bigl(1-\frac1{p^2}\Bigr)\right)^{(1-\chi_4(a))/2}.
\]
In particular, the extra factors $(1-1/p^2)^{(1-\chi_4(a))/2}$ occur only when $\chi_4(a)=-1$ and are precisely compensated by the denominator $\prod (1-1/p^2)^{-1}$ coming from Lemma~\ref{lem:factor}. Insert this identity into the previous display, multiply over all $a\in S$, and simplify. One obtains
\[
P_S(x)=
\prod_{a\in S}
\left(\prod_{\substack{p\le x\\ p\equiv a\ (q)}}\Bigl(1-\frac1p\Bigr)\right)^{2\chi_4(a)}
\cdot
\prod_{a\in S}
\left(\prod_{\substack{p\le x\\ p\equiv a\ (q)}}\Bigl(1-\frac1{p^2}\Bigr)\right)^{-\chi_4(a)}
\cdot(1+o(1)).
\]
By the definitions of $C(q,a)$ and $D(q,a)$,
\[
\prod_{\substack{p\le x\\ p\equiv a\ (q)}}\Bigl(1-\frac1p\Bigr)
=
C(q,a)\,(\log x)^{-1/\varphi(q)}(1+o(1))
\]
and
\[
\prod_{\substack{p\le x\\ p\equiv a\ (q)}}\Bigl(1-\frac1{p^2}\Bigr)
=
D(q,a)(1+o(1)).
\]
Substituting these asymptotics yields
\[
P_S(x)=
\left(\prod_{a\in S}\left(\frac{C(q,a)^2}{D(q,a)}\right)^{\chi_4(a)}\right)
(\log x)^{-2\mu(S)/\varphi(q)}(1+o(1)).
\]
This proves the asserted formula for $K(q,S)$. If $\mu(S)=0$, the logarithmic factor disappears and the finite nonzero limit equals the stated product.
\end{proof}

The following observation explains why the constants may be written in terms of Dirichlet $L$-values.
\begin{remark}
Let $G=(\mathbb Z/q\mathbb Z)^\times$. By character orthogonality on $G$, the indicator of a subset $S\subset G$ admits a Fourier expansion
\[
1_S(b)=\frac{1}{|G|}\sum_{\chi\bmod q}\widehat{1_S}(\chi)\chi(b),
\qquad
\widehat{1_S}(\chi)=\sum_{a\in S}\overline{\chi(a)}.
\]
Inserting this expansion into the relevant Euler products yields expressions for $\log C(q,a)$ and $\log D(q,a)$ as finite rational linear combinations of $\log L(1,\chi)$ and $\log L(2,\chi)$.
\end{remark}

\section{Examples}

In the first example we evaluate the full odd-prime product explicitly; in the second we compute a nontrivial thinned product in closed form; the third example illustrates how the accelerated character formula in \ref{app:computation} resolves the first case not covered by an elementary closed form.

\vskip 5pt\noindent {\bf Example 1.} Let $q=4$, and take $S=(\mathbb Z/4\mathbb Z)^\times=\{1,3\}$. Then $\mu(S)=0$, and
\[
\prod_{p\ \mathrm{odd}} A(p)=2.
\]
Indeed, by Lemma~\ref{lem:factor},
\[
\prod_{p\ \mathrm{odd}} A(p)
=
\frac{\left(\prod_{p\ \mathrm{odd}}\left(1-\frac{\chi_4(p)}{p}\right)\right)^2}
{\prod_{p\ \mathrm{odd}}\left(1-\frac{1}{p^2}\right)}.
\]
The Euler product for $L(1,\chi_4)$ gives
\[
\prod_{p\ \mathrm{odd}}\left(1-\frac{\chi_4(p)}{p}\right)=\frac{1}{L(1,\chi_4)}=\frac{4}{\pi}.
\]
Moreover,
\[
\prod_{p\ \mathrm{odd}}\left(1-\frac{1}{p^2}\right)
=\frac{1}{\zeta(2)}\cdot\frac{1}{1-2^{-2}}
=\frac{8}{\pi^2}.
\]
Hence, $\prod_{p\ \mathrm{odd}} A(p)=(4/\pi)^2/(8/\pi^2)=2$.

\vskip 5pt\noindent {\bf Example 2.} Let $q=8$ and $S=\{\pm 1\}$. Let $\chi_8$ be the real character modulo $8$ defined by $\chi_8(p)=1$ for $p\equiv \pm 1\pmod 8$ and $\chi_8(p)=-1$ for $p\equiv \pm 3\pmod 8$ for odd primes $p$. Write
\[
P_+(x)=\prod_{\substack{p\le x\\ p\ \mathrm{odd}\\ p\equiv \pm 1\ (8)}} A(p),
\qquad
P_-(x)=\prod_{\substack{p\le x\\ p\ \mathrm{odd}\\ p\equiv \pm 3\ (8)}} A(p),
\]
and denote $P_+=\lim_{x\to\infty}P_+(x)$ and $P_-=\lim_{x\to\infty}P_-(x)$. Set $S_+=\{\pm 1\}$ and $S_-=\{\pm 3\}$ in $(\mathbb Z/8\mathbb Z)^\times$. By Corollary~\ref{cor:criterion} both limits $P_+$ and $P_-$ exist in $(0,\infty)$, since $\mu(S_+)=\mu(S_-)=0$. Then $P_+P_-=\prod_{p\ \mathrm{odd}}A(p)=2$ by Example 1. Moreover,
\[
\frac{P_+}{P_-}=\prod_{p\ \mathrm{odd}} A(p)^{\chi_8(p)}.
\]
Set $\psi=\chi_4\chi_8$, which is the real character modulo $8$ with $\psi(p)=1$ for $p\equiv 1,3\pmod 8$ and $\psi(p)=-1$ for $p\equiv 5,7\pmod 8$. A check on residue classes shows
\[
A(p)^{\chi_8(p)}=\frac{p-\psi(p)}{p+\psi(p)}
\qquad (p\ \mathrm{odd}).
\]
Using Lemma~\ref{lem:factor} with $\psi$ in place of $\chi_4$ we obtain
\[
\prod_{p\ \mathrm{odd}} A(p)^{\chi_8(p)}
=
\frac{\left(\prod_{p\ \mathrm{odd}}(1-\psi(p)/p)\right)^2}{\prod_{p\ \mathrm{odd}}(1-1/p^2)}
=
\left(\frac{1}{L(1,\psi)}\right)^2\cdot\frac{\pi^2}{8}.
\]
The classical value $L(1,\psi)=\pi/(2\sqrt2)$ gives $P_+/P_-=1$. Therefore, $P_+=P_-=\sqrt2$, hence
\[
\prod_{\substack{p\ \mathrm{odd}\\ p\equiv \pm 1\ (8)}} A(p)=\sqrt{2}.
\]
We use the classical evaluation $L(1,\psi)=\pi/(2\sqrt{2})$ for the real primitive character $\psi$ modulo $8$; see, for instance, Davenport~\cite[Chapter 6]{Davenport}.

\vskip 5pt\noindent {\bf Example 3.} Let $q=16$ and $S=\{\pm 1\}=\{1,15\}$. Here $\mu(S)=0$. Theorem~\ref{thm:constant} gives
\[
K(16,S)=\frac{C(16,1)^2D(16,15)}{C(16,15)^2D(16,1)}.
\]
A direct computation of the partial products
\[
P_S(x)=\prod_{\substack{p\le x\\ p\ \mathrm{odd}\\ p\equiv \pm 1\ (16)}} A(p)
\]
is numerically misleading, because the convergence is slow and not monotone:
\[
P_S(10^6)=1.0000256406\ldots,
\qquad
P_S(10^7)=0.9999838172\ldots.
\]
Using instead the character-sum evaluation described in \ref{app:computation}, with working precision $100$ decimal digits and truncation parameter $n_{\max}=60$, one obtains
\[
\log K(16,\{\pm1\})
= -2.56197851408213398623664392190796013523644611\cdot 10^{-56}.
\]
Thus, $\log K(16,\{\pm1\})=0$ to at least $40$ decimal places, and hence
\[
K(16,\{\pm1\})=1.000000000000000000000000000000000000000\ldots .
\]
This explains why the direct partial products oscillate very close to $1$, while still being inadequate as a reliable method for determining the constant.

\section{Concluding Remarks}

The arguments above apply verbatim to any modulus $q$ with $4$ dividing $q$. Indeed, for such $q$ the function $\chi_4$ is constant on each reduced residue class modulo $q$, and the proofs of Theorems~\ref{thm:asymptotic} and~\ref{thm:constant} rely only on Mertens-type asymptotics in arithmetic progressions and on the identity in Lemma~\ref{lem:factor}. In particular, for any $S\subset(\mathbb Z/q\mathbb Z)^\times$ the product $P_S(x)$ admits the logarithmic asymptotic of Theorem~\ref{thm:asymptotic}, and it converges to a finite nonzero limit precisely when the balance condition $\mu(S)=0$ holds.
For $q=2^m$ there are infinitely many balanced selections $S$, hence infinitely many associated limits $K(q,S)$. A simple family is given by $S=\{\pm 1 \pmod{2^m}\}$, and more generally by any union of residue classes containing equally many classes congruent to $1$ and to $3$ modulo $4$. By character orthogonality one may express the constants $K(q,S)$ in terms of finitely many Dirichlet values $L(1,\chi)$ and $L(2,\chi)$, but we do not pursue explicit evaluations beyond the examples given above.

\newpage
\appendix
\renewcommand{\thesection}{Appendix \Alph{section}}

\section{High-Precision Computation of the Constants}\label{app:computation}

The direct truncation of the prime product defining $P_S(x)$ is useful as a consistency check, but it is not an efficient method for determining the limiting constants. We therefore evaluate these constants by means of a character-sum formula obtained by combining the formula of Languasco and Zaccagnini for the Mertens product constants $C(q,a)$ with the corresponding formula of Languasco and Moree for $\zeta_{q,a}(2)$; see \cite[(5)]{LZII} and \cite[(5.16)]{LM}. The background identities for these constants are given in \cite{LZ}.
Throughout this appendix let $q=2^m$ and let $\chi_0$ denote the principal character modulo $q$. To avoid a conflict with the notation $\mu(S)$ used in the main text, the M\"obius function is denoted by $\Mob$. For $S\subset(\mathbb Z/q\mathbb Z)^\times$ set
\[
c_S(\chi)=\sum_{a\in S}\chi_4(a)\chi(a).
\]
Then the constant in Theorem~\ref{thm:constant} may be evaluated from
\begin{align}
\log K(q,S)
={}&2^{1-m}\mu(S)
\left(
2(\log 2-\gamma)
+
\sum_{k\ge1}\frac1k\sum_{j\ge1}\frac{\Mob(j)}{j}
\log L(2kj,\chi_0)
\right) \notag\\
&\quad+
2^{1-m}
\sum_{\substack{\chi\bmod q\\ \chi\ne\chi_0}}
 c_S(\chi)
\sum_{k\ge1}\frac1k\sum_{j\ge1}\frac{\Mob(j)}{j}
\notag\\
&\quad\quad\times
\left(
\log L(2kj,\chi^j)-2\log L(kj,\chi^j)
\right).
\label{eq:logK-character}
\end{align}
Here $\chi^j$ denotes the pointwise $j$-th power of $\chi$. In the balanced case $\mu(S)=0$, the first line of Equation~\eqref{eq:logK-character} vanishes, and the computation involves only nonprincipal characters in the outer character sum.
In practice the sums over $j$ may be restricted to squarefree $j$, since $\Mob(j)=0$ otherwise. For the numerical implementation it is convenient to reorganize the double sum by putting $n=kj$:
\[
\sum_{k,j\ge1}\frac{\Mob(j)}{kj}F(kj,j)
=
\sum_{n\ge1}\frac1n\sum_{j\mid n}\Mob(j)F(n,j).
\]
In the principal-character part this immediately collapses to the term $n=1$, because $\sum_{j\mid n}\Mob(j)=0$ for $n>1$. The computations below were carried out in Python using arbitrary-precision arithmetic from \texttt{mpmath} \cite{Mpmath,Python}. Dirichlet $L$-values were evaluated via Hurwitz zeta functions, with the usual digamma formula for $L(1,\chi)$ when $\chi$ is nonprincipal. For speed, all characters modulo $2^m$ were evaluated by using the decomposition
\[
(\mathbb Z/2^m\mathbb Z)^\times=\{\pm1\}\times\langle5\rangle.
\]
The scripts used for the tables are available from the author's webpage \cite{WinklerCode2026}. Languasco's publicly available PARI/GP programs provide an independent reference implementation for the same underlying constants \cite{LanguascoPrograms}.
Table~\ref{tab:accelerated-values} gives the computed values for the family $S_m=\{1,2^m-1\}$. Table~\ref{tab:stability} gives the corresponding stability check.

\begin{table}[!ht]
\centering
\scriptsize
\resizebox{\textwidth}{!}{%
\begin{tabular}{@{}r r c c@{}}
\toprule
$m$ & $q=2^m$ & $\log K(q,S_m)$ & $K(q,S_m)$ \\
\midrule
$4$  & $16$   & $|\log K|<3\cdot 10^{-56}$ &
$1.000000000000000000000000000000000000000$ \\
$5$  & $32$   &
$0.055785887828552438941573772577458625844$ &
$1.057371263440564119535037000028605726981$ \\
$6$  & $64$   &
$0.000372834967240502990072479904582710305$ &
$1.000372904478835417597736161858272147763$ \\
$7$  & $128$  &
$0.008748938687609497210384882738233790907$ &
$1.008787322509263329452892922676343582795$ \\
$8$  & $256$  &
$-0.008048465895392386947216571944447322214$ &
$0.991983836287144624613877998261566861441$ \\
$9$  & $512$  &
$0.0008587933600776863752417291695811209997$ &
$1.000859162228681755059086980611386764540$ \\
$10$ & $1024$ &
$0.0005667492584726388994062662029803796723$ &
$1.000566909891178350882758632030919205021$ \\
\bottomrule
\end{tabular}%
}
\caption{Computed values of $\log K(q,S_m)$ and $K(q,S_m)$ for $S_m=\{1,2^m-1\}\subset(\mathbb Z/2^m\mathbb Z)^\times$. The computations used $100$ decimal digits and $n_{\max}=60$ in the truncated character sum. All values are rounded to the last displayed decimal place.}
\label{tab:accelerated-values}
\end{table}

The first two exact cases, $K(4,\{1,3\})=2$ and $K(8,\{1,7\})=\sqrt2$, are treated in Examples 1 and 2. Table~\ref{tab:accelerated-values} begins with the first case not covered by an elementary closed form. The row $m=4$ resolves the case $q=16$, $S=\{\pm1\}$ from Example 3: the raw partial products alone do not determine the constant reliably, whereas the character formula gives $K(16,S)=1$ to at least $40$ decimal places.

\begin{table}[!ht]
\centering
\scriptsize
\resizebox{\textwidth}{!}{%
\begin{tabular}{@{}r r c c@{}}
\toprule
$m$ & $q$ & $\log K$ at $n_{\max}=30$ & $\log K$ at $n_{\max}=60$ \\
\midrule
$4$  & $16$   &
$-5.6599497521\cdot 10^{-29}$ &
$-2.5619785141\cdot 10^{-56}$ \\
$5$  & $32$   &
$0.055785887828552438941573772577458625844$ &
$0.055785887828552438941573772577458625844$ \\
$6$  & $64$   &
$0.000372834967240502990072479904582710305$ &
$0.000372834967240502990072479904582710305$ \\
$7$  & $128$  &
$0.008748938687609497210384882738233790907$ &
$0.008748938687609497210384882738233790907$ \\
$8$  & $256$  &
$-0.008048465895392386947216571944447322214$ &
$-0.008048465895392386947216571944447322214$ \\
$9$  & $512$  &
$0.0008587933600776863752417291695811209997$ &
$0.0008587933600776863752417291695811209997$ \\
$10$ & $1024$ &
$0.0005667492584726388994062662029803796723$ &
$0.0005667492584726388994062662029803796723$ \\
\bottomrule
\end{tabular}%
}
\caption{Stability check for the character-sum computation of Table~\ref{tab:accelerated-values}. The working precision was $100$ decimal digits. All displayed values are real; the imaginary parts arising from the complex-valued character computation were below $10^{-102}$ in absolute value and are omitted.}
\label{tab:stability}
\end{table}

\section{Comparison with Direct Partial Products}\label{app:direct-products}

For comparison, Table~\ref{tab:direct-products} records direct truncations for the balanced thinnings $p\equiv \pm 1\pmod{2^m}$. These values should not be used as the primary numerical determination of the constants; their role is only to check qualitative consistency with Table~\ref{tab:accelerated-values}. The agreement is visible, but the convergence is much slower than the computation based on Equation~\eqref{eq:logK-character}. For example, the direct value for $q=32$ at $x=10^7$ is still only close to the accelerated value $1.05737126344056\ldots$.

\begin{table}[!ht]
\centering
\small
\begin{tabular}{@{}r c c c c@{}}
\toprule
$x$ & $q=8$ & $q=16$ & $q=32$ & $q=64$ \\
\midrule
$10^4$ & 1.4113275632 & 0.9999436166 & 1.0579775029 & 1.0006591649 \\
$10^5$ & 1.4130836158 & 0.9995823760 & 1.0572678232 & 1.0001694785 \\
$10^6$ & 1.4141098422 & 1.0000256406 & 1.0573224354 & 1.0003905341 \\
$10^7$ & 1.4141721971 & 0.9999838172 & 1.0573538661 & 1.0003650135 \\
\bottomrule
\end{tabular}%
\caption{Direct partial products for balanced thinnings $p\equiv \pm 1\pmod q$. The columns report $\prod_{p\le x,\ p\equiv \pm1\ (q)}A(p)$ for the indicated values of $q$.}
\label{tab:direct-products}
\end{table}

Now consider $q=16$ and $S=\{3,11\}$. Both residue classes are congruent to $3$ modulo $4$, hence
\[
\mu(S)=\chi_4(3)+\chi_4(11)=-2,
\]
and the predicted exponent is
\[
-\frac{2\mu(S)}{\varphi(16)}=\frac12.
\]
The character formula in Equation~\eqref{eq:logK-character}, evaluated with $100$ decimal digits and $n_{\max}=160$, gives
\[
K(16,\{3,11\})=1.421545527665792738447977085120563275809\ldots .
\]
Thus,
\[
P_{\{3,11\}}(x)\sim 1.421545527665792738447977085120563275809\ldots\,(\log x)^{1/2}.
\]
The direct normalized values in Table~\ref{tab:unbalanced-direct} are consistent with this value.

\begin{table}[!ht]
\centering
\small
\begin{tabular}{@{}r c c@{}}
\toprule
$x$ & $P_{\{3,11\}}(x)$ & $(\log x)^{-1/2}P_{\{3,11\}}(x)$ \\
\midrule
$10^4$ & 4.3149324585 & 1.4217923138 \\
$10^5$ & 4.8235475228 & 1.4215878898 \\
$10^6$ & 5.2838015898 & 1.4215529197 \\
$10^7$ & 5.7070707662 & 1.4215311910 \\
$10^8$ & 6.1011757220 & 1.4215452732 \\
\bottomrule
\end{tabular}
\caption{Direct values for the unbalanced example $q=16$ and $S=\{3,11\}$.}
\label{tab:unbalanced-direct}
\end{table}

\end{document}